\documentclass{amsart}
\usepackage{float}
\usepackage[ruled,linesnumbered]{algorithm2e}
\usepackage{algorithmicx}
\usepackage{algpseudocode}
\usepackage{tikz}

\tikzstyle{vertex}=[auto=left,circle,fill=black!0,minimum size=0pt,inner sep=0pt]

\newtheorem{theorem}{Theorem}[section]
\newtheorem{lemma}[theorem]{Lemma}
\newtheorem{corollary}[theorem]{Corollary}

\newtheorem{proposition}[theorem]{Proposition}
\theoremstyle{definition}  
\newtheorem{definition} [theorem] {Definition}

\newtheorem{remark} [theorem] {Remark}

\newcommand{\Diam}{\mathrm{diam}}
\newcommand{\Link}{N_G}
\newcommand{\Linkp}{N_G^>}

\numberwithin{equation}{section}

\newcommand{\fignum}[1]{#1}

\begin{document}

\title{On Closed Graphs I}

\author{David A.\ Cox}
\address{Department of Mathematics and Statistics, Amherst
College, Amherst, MA 01002-5000, USA}
\email{dacox@amherst.edu}

\author{Andrew Erskine}
\address{Department of Mathematics and Statistics, Amherst
College, Amherst, MA 01002-5000, USA}
\email{aperskine@gmail.com}

\begin{abstract}
A graph is closed when its vertices have a labeling by $[n]$ with a
certain property first discovered in the study of binomial edge
ideals.  In this article, we prove that a connected graph has a closed
labeling if and only if it is chordal, claw-free, and has a property we
call \emph{narrow}, which holds when every vertex is distance at most
one from all longest shortest paths of the graph.  
 \end{abstract}

\keywords{closed graph, chordal graph, binomial edge ideal} 
\subjclass[2010]{05C75 (primary), 05C25, 05C78 13P10 (secondary)}

\maketitle

\section{Introduction}
\label{intro}

In this paper, $G$ will be a simple graph with vertex
set $V(G)$ and edge set $E(G)$. 

\begin{definition}
\label{closeddef}
A \emph{labeling} of $G$ is a bijection $V(G) \simeq [n] =
\{1,\dots,n\}$, and given a labeling, we typically assume $V(G) =
   [n]$.  A labeling is \emph{closed} if whenever we have distinct edges
   $\{j,i\}, \{i,k\} \in E(G)$ with either $j > i < k$ or $j < i > k$,
   then $\{j,k\} \in E(G)$.  Finally, a graph is \emph{closed} if it
   has a closed labeling.
\end{definition}

A labeling of $G$ gives a direction to each edge $\{i,j\} \in E(G)$
where the arrow points from $i$ to $j$ when $i < j$, i.e., the arrow
points to the bigger label.  The following picture illustrates
what it means for a labeling to be closed:
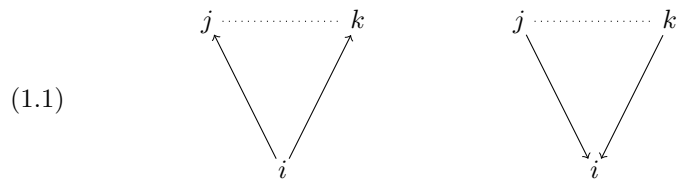
\begin{equation}
\label{closedpicture}
\begin{array}{ccc}
\begin{tikzpicture}
  \node[vertex] (n1)  at (2,1) {$i\rule[-2.5pt]{0pt}{10pt}$};
  \node[vertex] (n2)  at (1,3)  {$\rule[-2.5pt]{0pt}{10pt}j$};
  \node[vertex] (n3) at (3,3) {$k\rule[-2.5pt]{0pt}{10pt}$};
 
  \foreach \from/\to in {n1/n2,n1/n3}
\draw[->] (\from)--(\to);;

 \foreach \from/\to in {n2/n3}
\draw[dotted] (\from)--(\to);;

\end{tikzpicture}&\hspace{30pt}&
\begin{tikzpicture}
  \node[vertex] (n1)  at (2,1) {$i\rule[-2.5pt]{0pt}{10pt}$};
  \node[vertex] (n2)  at (1,3)  {$j\rule[-2.5pt]{0pt}{10pt}$};
  \node[vertex] (n3) at (3,3) {$k\rule[-2.5pt]{0pt}{10pt}$};
 
  \foreach \from/\to in {n2/n1,n3/n1}
\draw[->] (\from)--(\to);;

 \foreach \from/\to in {n2/n3}
\draw[dotted] (\from)--(\to);;

\end{tikzpicture}
\end{array}
\end{equation}
Whenever the arrows point away from $i$ (as on the left) or
towards $i$ (as on the right), closed means that $j$ and $k$
are connected by an edge.

Closed graphs were first encountered in the study of binomial edge
ideals.  The \emph{binomial edge ideal} of a labeled graph $G$ is the
ideal $J_G$ in the polynomial ring
$\mathsf{k}[x_1,\dots,x_n,y_1,\dots,y_n]$ ($\mathsf{k}$ a field)
generated by the binomials
\[
f_{ij} = x_iy_j - x_jy_i
\]
for all $i,j$ such that $\{i,j\} \in E(G)$ and $i < j$.  A key result,
discovered independently in \cite{H} and \cite{O}, is that the above
binomials form a Gr\"obner basis of $J_G$ for lex order with $x_1 >
\cdots > x_n > y_1 > \cdots > y_n$ if and only if the labeling is
closed.  The name ``closed'' was introduced in \cite{H}.

Binomial edge ideals are explored in \cite{E} and \cite{S}, and a
generalization is studied in \cite{R}.  The paper \cite{CR}
characterizes closed graphs using the clique complex of $G$, and
closed graphs also appear in \cite{E3,E4,E2}.

The goal of this paper is to characterize when a graph $G$ has a
closed labeling in terms of properties that can be seen directly from
the graph.  Our starting point is the following result proved in
\cite{H}. 

\begin{proposition} 
\label{Hprop}
Every closed graph is chordal and claw-free.
\end{proposition}

``Claw-free'' means that $G$ has no induced subgraph of the form
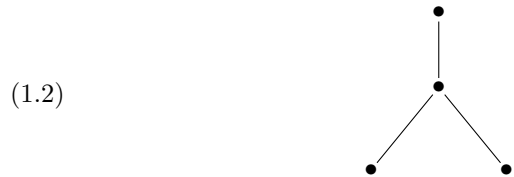
\begin{equation}
\label{ex1}
\begin{array}{c}
\begin{tikzpicture}
  \node[vertex] (k) at (3,6) {$\bullet$};
  \node[vertex] (j)  at (2.1,3.9){$\bullet$};
  \node[vertex] (l) at (3.9,3.9) {$\bullet$}; 
  \node[vertex] (i) at (3,5){$\bullet$};

  \foreach \from/\to in {i/l,i/k,i/j}
    \draw (\from) -- (\to);

\end{tikzpicture}
\end{array}
\end{equation}

Besides being chordal and claw-free, closed graphs also have a
property called \emph{narrow}.  The \emph{distance} $d(v,w)$ between
vertices $v,w$ of a connected graph $G$ is the length of the shortest
path connecting them, and the \emph{diameter} of $G$ is
$\mathrm{diam}(G) = \max\{d(v,w) \mid v,w \in E(G)\}$.  Given vertices
$v,w$ of $G$ satisfying $d(v,w) = \mathrm{diam}(G)$, a shortest path
connecting $v$ and $w$ is called a \emph{longest shortest path} of
$G$.

\begin{definition} 
\label{narrowdef}
A connected graph $G$ is \emph{narrow} if for every $v \in V(G)$ and
every longest shortest path $P$ of $G$, either $v \in V(P)$ or
$\{v,w\} \in E(G)$ for some $w \in V(P)$. 
\end{definition}

Thus a connected graph is narrow if every vertex is distance at most
one from every longest shortest path.  Here is a graph that is chordal
and claw-free but not narrow:
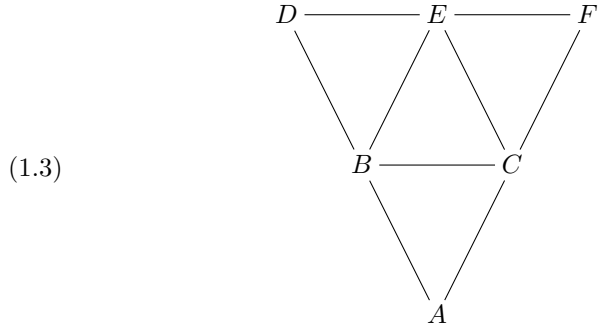
\begin{equation}
\label{ex2}
\begin{array}{c}
\begin{tikzpicture}
  \node[vertex] (n1)  at (3,1) {$A\rule[-2pt]{0pt}{10pt}$};
  \node[vertex] (n2)  at (2,3)  {$B\rule[-2pt]{0pt}{10pt}$};
  \node[vertex] (n3) at (4,3) {$C\rule[-2pt]{0pt}{10pt}$};
 
  \node[vertex] (n4) at (3,5){$E\rule[-2pt]{0pt}{10pt}$};
  \node[vertex] (n5) at (5,5){$F\rule[-2pt]{0pt}{10pt}$};
  \node[vertex] (n6) at (1,5){$D\rule[-2pt]{0pt}{10pt}$};

  \foreach \from/\to in {n1/n2,n1/n3,n2/n3,n2/n4, n3/n4, n3/n5,
    n4/n5,n2/n6,n4/n6} 
\draw (\from)--(\to);;

\end{tikzpicture}
\end{array}
\end{equation}
Narrowness fails because $D$ is distance two from the
longest shortest path $ACF$.

We can now state the main result of this paper.

\begin{theorem} 
\label{mainthm}
A connected graph is closed if and only if it is chordal, claw-free,
and narrow.
\end{theorem}

This theorem is cited in \cite{E3,E4,E2}.  Since a graph is closed if
and only if its connected components are closed \cite{CR}, we get the
following corollary of Theorem~\ref{mainthm}.

\begin{corollary}
\label{cormainthm}
A graph is closed if and only if it is chordal, claw-free, and its
connected components are narrow. 
\end{corollary}

The independence of the three conditions (chordal, claw-free, narrow)
is easy to see.  The graph \eqref{ex1} is chordal and narrow but not
claw-free, and the graph \eqref{ex2} is chordal and claw-free but not
narrow.  Finally, the $4$-cycle
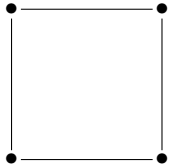
\[
\begin{tikzpicture}
  \node[vertex] (A)  at (2,1) {$\bullet$};
  \node[vertex] (B)  at (4,1)  {$\bullet$};
  \node[vertex] (C) at (4,3) {$\bullet$};
 
  \node[vertex] (D) at (2,3) {$\bullet$};

  \foreach \from/\to in {A/B,B/C,C/D,D/A}
\draw (\from)--(\to);;

\end{tikzpicture}
\]
is claw-free and narrow but not chordal.

The paper is organized as follows.  In Section~\ref{properties} we
recall some known properties of closed graphs and prove some new ones,
and in Section~\ref{algorithm} we introduce an algorithm for labeling
connected graphs.  Section~\ref{characterize} uses the algorithm  to
prove Theorem~\ref{mainthm}.   

In a subsequent paper \cite{closed2} we will explore further
properties of closed graphs.

\section{Properties of Closed Labelings}
\label{properties}

\subsection{Directed Paths} A path in a graph $G$ is 
$P = v_0v_1\cdots v_{\ell-1}v_\ell$ where $\{v_j,v_{j+1}\} \in E(G)$
for $j = 0,\dots,\ell-1$.  A single vertex is regarded as a path of
length zero.  When $G$ is labeled, we assume as usual that $V(G) =
[n]$.  Then a path $P = i_0i_1\cdots i_{\ell-1}i_\ell$ is
\emph{directed} if either $i_j < i_{j+1}$ for all $j$ or $i_j >
i_{j+1}$ for all $j$.  Here is a result from \cite{H}.

\begin{proposition} 
\label{directed}
A labeling on a graph $G$ is closed if and only if for all vertices
$i,j \in V(G) = [n]$, all shortest paths from $i$ to $j$ are directed.
\end{proposition}




\subsection{Neighborhoods and Intervals} Given a vertex $v \in V(G)$, the
\emph{neighborhood} of $v$ in $G$ is
\[
N_G(v) = \{ w \in V(G) \mid \{v,w\} \in E(G)\}.
\]
When $G$ is labeled and $i \in V(G) = [n]$, we have a disjoint union
\[
N_G(i) = N_G^>(i) \cup N_G^<(i),
\]
where
\[
N_G^>(i) = \{j \in N_G(i) \mid j > i\},\
N_G^<(i) = \{j \in N_G(i) \mid j < i\}.
\]
This is the notation used in \cite{CR}, where 
%
%
it is shown that a labeling is closed if and only if $N_G^>(i)$ and
$N_G^<(i)$ are complete for all $i \in V(G) = [n]$.

Vertices $i,j \in [n]$ with $i \le j$ give the \emph{interval} $[i,j]
= \{k \in [n] \mid i \le k \le j\}$.  Here is a characterization of
when a labeling of a connected graph is closed.

\begin{proposition}
\label{nbdinterval}
A labeling on a connected graph $G$ is closed if and only if for all
$i \in V(G) = [n]$, $N_G^>(i)$ is complete and equal to
$[i+1,i+r]$, $r = |N_G^>(i)|$.
\end{proposition}

\begin{proof}
Assume that the labeling is closed.  Then Definition~\ref{closeddef}
easily implies that $N_G^>(i)$ is complete.  It remains to show that
$N_G^>(i)$ is an interval of the desired form.

Pick $j \in N_G^>(i)$ and $k\in[n]$ with $i<k<j$.  A shortest path
$P=i_0i_1 i_2\cdots i_m$ from $i = i_0$ to $k = i_m$ is directed by
Proposition~\ref{directed}.  Since $i<k$, we have
$i=i_0<i_1<i_2<\cdots<i_m=k$.  Thus $i_1\in \Linkp(i)$ and hence
$\{i_1,j \}\in E(G)$ since $\Linkp(i)$ is complete. Since $i_1<j$, we
have $j\in \Linkp(i_1)$.

We now prove by induction that $j \in \Linkp(i_u)$ for all $u =
1,\dots,m$.  The base case is proved in the previous paragraph.  Now
assume $j\in \Linkp(i_u)$.  Then $\{j,i_{u+1}\}\in E(G)$ since
$\{i_u,i_{u+1}\}\in E(G)$ and the labeling is closed.  This completes
the induction.  Since $k = i_m$, it follows that $j\in \Linkp(k)$.
Then we have $\{i,j\},\{k,j\}\in E(G)$ with $i<j>k$.  Thus $\{i,k\}\in
E(G)$ since the labeling is closed, so $k\in \Linkp(i)$ since $i<k$.
Hence $N_G^>(i)$ is an interval of the desired form.

Conversely, suppose that $N_G^>(i)$ is complete and $N_G^>(i) =
[i+1,\dots,i+r]$, $r = |N_G^>(i)|$, for all $i \in V(G)$.  Take
$\{j,i\}, \{i,k\} \in E(G)$ with $j > i < k$ or $j < i > k$.  The
former implies $\{j,k\} \in E(G)$ since $\Linkp(i)$ is complete.  For
the latter, assume $j<k$.  Then $j < k < i$ with $i \in \Linkp(j)$.
Since $\Linkp(j)$ is an interval containing $j+1$ and $i$, $\Linkp(j)$
also contains $k$.  Hence $\{j,k\} \in E(G)$.
\end{proof}  

\subsection{Layers} The following subsets of $V(G)$ will play a key
role in what follows.

\begin{definition} 
\label{layerdef}
Let $G$ be a connected graph labeled so that $V(G)
= [n]$.  Then the \emph{$N^{\mathit{th}}$ layer of $G$} is the set
\[
L_N = \{i \in [n] \mid d(i,1) = N\}.
\]
\end{definition}

Thus $L_N$ consists of all vertices that are distance $N$ from the
vertex $1$.  Note that $L_0 = \{1\}$ and $L_1 = \Link(1) = \Linkp(1)$.
Furthermore, since $G$ is connected, we have a disjoint union
\[
V(G) = L_0 \cup L_1 \cup \cdots \cup L_h,
\]
where $h = \max\{d(i,1) \mid i \in [n]\}$.  We omit the easy proof of the
following lemma.

\begin{lemma}
\label{layerlem}
Let $G$ be a connected graph labeled so that $V(G) = [n]$.  Then:
\begin{enumerate}
\item If $i \in L_N$ and $\{i,j\} \in E(G)$, then $j \in L_{N-1}$,
  $L_N$, or $L_{N+1}$.
\item If $P$ is a path in $G$ connecting $i \in L_N$ to $j \in L_M$
  with $N \le M$, then for every integer $N \le m \le M$, there exists
  $k \in V(P)$ with $k \in L_m$.
\end{enumerate}
\end{lemma}

\begin{proposition}
\label{layerprop}
Let $G$ be a connected graph with a closed labeling satisfying $V(G) =
[n]$.  Then:
\begin{enumerate}
\item Each layer $L_N$ is complete.
\item If $d = \max\{L_N\}$, then $L_{N+1} = \Linkp(d)$.
\end{enumerate}
\end{proposition}

\begin{proof}
We first show that
\begin{equation}
\label{nextlayer}
r \in L_M,\ s \in L_{M+1},\ \{r,s\} \in E(G) \Longrightarrow r < s.
\end{equation}
To see why, take a shortest path from 1 to $r \in L_M$.  This path has
length $M$, so appending the edge $\{r,s\}$ gives a path of length
$M+1$ to $s$.  Since $s \in L_{M+1}$, this is a shortest path and
hence is directed by Proposition~\ref{directed}.  Thus $r<s$.

For (1), we use
induction on $N \ge 0$.  The base case is trivial since $L_0=\{1\}$.
Now assume $L_N$ is complete and take $i,j\in L_{N+1}$ with $i\neq j$.
A shortest path $P_1$ from $1$ to $i\in L_{N+1}$ has a vertex $k \in
L_N$ adjacent to $i$, and a shortest path $P_2$ from $1$ to $j\in
L_{N+1}$ has a vertex $l\in L_{N}$ adjacent to $j$.  Then $k<i$ and
$l<j$ by \eqref{nextlayer}.

If $k=l$, then $i>k<j$, which implies $\{i,j\}\in E(G)$ since the
labeling is closed.  If $k\neq l$, then $\{l,k\}\in E(G)$ since $L_N$
is complete.  Assume $l>k$.  Then $l>k<i$ and closed imply $\{l,i\}\in
E(G)$.  Since $l \in L_N$ and $i \in L_{N+1}$, we have $l<i$ by
\eqref{nextlayer}.  Then $i>l<j$ and closed imply $\{i,j\}\in E(G)$.
Hence $L_{N+1}$ is complete.

We now turn to (2).  To prove $L_{N+1} \subseteq \Linkp(d)$, $d =
\max\{L_N\}$, take $i \in L_{N+1}$.  A shortest path from $1$ to $i$
will have a vertex $k\in L_N$ such that $\{k,i\}\in E(G)$.  Then $k<i$
by \eqref{nextlayer}, hence $i\in \Linkp(k)$.  Also, $k\leq d$ since
$d=\max(L_N)$.  If $k=d$, then $i\in \Linkp(d)$.  If $k<d$, then
$\{k,d\}\in E(G)$ since $L_N$ is complete.  Then $i>k<d$ and closed
imply $\{d,i\}\in E(G)$, and then $d<i$ by \eqref{nextlayer}.  Thus
$i\in \Linkp(d)$.

To prove the opposite inclusion, take $i\in \Linkp(d)$.  Since
$\{d,i\}\in E(G)$ and $d \in L_N$, we have $i \in L_M$ for $M = N-1,N,
N+1$ by Lemma~\ref{layerlem}.  If $i\in L_{N-1}$, then
\eqref{nextlayer} would imply $i < d$, contradicting $i \in
\Linkp(d)$.  If $i\in L_N$, then $i \le \max\{L_N\} = d$, again
contradicting $i \in \Linkp(d)$.  Hence $i\in L_{N+1}$.
\end{proof}

\subsection{Longest Shortest Paths} 
When the labeling of a connected graph is closed, the diameter of the
graph determines the number of layers as follows.

\begin{proposition}
\label{diameters}
Let $G$ be a connected graph with a closed labeling.  Then:
\begin{enumerate}
\item $\Diam(G)$ is the largest integer $h$ such that $L_h\neq
  \emptyset$.
\item If $P$ is a longest shortest path of $G$, then one endpoint of
  $P$ is in $L_0$ or $L_1$ and the other is in $L_h$, where
  $h=\Diam(G)$.
\end{enumerate}
\end{proposition}

\begin{proof}
For (1), let $h$ be the largest integer with $L_h\neq \emptyset$.
Since points in $L_h$ have distance $h$ from $1$, we have $h \le
\Diam(G)$.

For the opposite inequality, it suffices to show that $d(i,j) \le h$
for all $i,j \in V(G)$ with $i \ne j$.  We can assume $G$ has more
than one vertex, so that $h\geq 1$.  Suppose $i\in L_N$ and $j\in L_M$
with $N\leq M$.  If $N=0$, then $i=1$ and $d(i,j) = d(1,j) = M \le h$
since $j \in L_M$.  Also, if $M=N$, then $i,j\in L_N$, so that $d(i,j) =
1 \le h$ since $L_N$ is complete by Proposition~\ref{layerprop}.
Finally, if $0<N<M$, let $d_u=\max(L_u)$ for each integer $u$.  By
Proposition~\ref{layerprop}, we know that $j \in \Linkp(d_{M-1})$.
Hence, if $i\neq d_N$, then $P=id_Nd_{N+1}\cdots d_{M-2}d_{M-1}j$ is a
path of length $M-N+1$. If $i=d_N$, then $P=id_{N+1}\cdots d_{M-1}j$
is a path of length $M-N$. Thus we have a path from $i$ to $j$ of
length at most $M-N+1$, so that $d(i,j) \le M-N+1 \le M \le h$.

For (2), let $i$ and $j$ be the endpoints of the longest shortest path
$P$ with $i\in L_N$, $j\in L_M$ and $N\leq M$.  If $0 < N < M$, then
the previous paragraph implies  
\[
\Diam(G) = d(i,j) \le M-N+1 \le M \le h = \Diam(G),
\]
which forces $N = 1$ (so $i \in L_1$) and $M = h$ (so $j \in L_h$).
The remaining cases $N = 0$ and $N = M$ are straightforward and are
left to the reader.   
\end{proof}

Recall from Definition~\ref{narrowdef} that a connected graph $G$ is
narrow when every vertex is distance at most one from every longest
shortest path.  Narrowness is a key property of connected closed
graphs. 

\begin{theorem}
\label{narrowthm}
Every connected closed graph is narrow.
\end{theorem}

\begin{proof}
Let $G$ be a connected graph with a closed labeling. Pick a vertex
$i\in V(G)$ and a longest shortest path $P$.  Since $G$ is connected,
$i\in L_N$ for some integer $N$.   By Proposition~\ref{diameters}, the
endpoints of $P$ lie in $L_0$ or $L_1$ and  $L_h$, $h = \Diam(G)$.
Then Lemma~\ref{layerlem} implies that $P$ has a vertex $i_M$ in $L_M$
for every $1 \le M \le h$.   

If $N \ge 1$, then either $i = i_N \in V(P)$ or $i \ne i_N$, in which
case $\{i,i_N\} \in E(G)$ since $L_N$ is complete by
Proposition~\ref{layerprop}.  On the other hand, if $N = 0$, then
$i\in L_0$, hence $i=1$.  Then $\{i,i_1\} = \{1,i_1\} \in E(G)$ since
$i_1 \in L_1 = \Link(1)$.  In either case, $i$ is distance at most one
from $P$.   
\end{proof}

\section{A Labeling Algorithm}
\label{algorithm}

 We introduce Algorithm~\ref{alg:Labeling}, which labels the vertices
 of a connected graph.  This algorithm will play a key role in the
 proof of Theorem~\ref{mainthm}.


\begin{algorithm}[ht]
\SetAlgoLined
 \KwIn{A connected graph $G$ with $n$ vertices$\strut$} \label{alg:input}
 \KwOut{A labeling $V(G)=[n]$ and a function $l:[n] \rightarrow
   \{0,\ldots,n-1\}$$\strut$} \label{alg:output} 
$i:=1$\;
$j:=0$\;
$v_0:=$ endpoint of a longest shortest path with minimal
 degree\;\label{alg:possible1} 
label $v_0$ as $i$\; \label{alg:label1}
$l(i):=j$\;\label{alg:function1}
$i:=i+1$\;
$J:=\{v_0\}$\hfill \tcc*[h]{initial value of set of labeled vertices}\;
$j:=j+1$\;\label{alg:labelj1}
 \While{$j\leq |V(G)|$}{ \label{alg:firstLoop}
$S:=\Link(j)\setminus J$\hfill \tcc*[h]{unlabeled vertices adjacent to $j$}\;\label{alg:Sset} 
  \While{$S\neq \emptyset$}{ \label{alg:secondLoop}
$v:=$ pick $v\in S$ such that $|\Link(v)\setminus J| = \min_{u\in
      S}\{|\Link(u)\setminus J|\}$\;\label{alg:beginSecondLoop} 
label $v$ as $i$\; \label{alg:label2}
$l(i):=j$\;\label{alg:function2}
$i:=i+1$\;\label{alg:iIncrement}
$J:=J \cup \{v\}$\hspace*{3pt} \tcc*[h]{add $v$ to labeled
  vertices}\; 
$S:=S\setminus \{v\}$\hspace*{3pt} \tcc*[h]{remove $v$ from unlabeled
  adjacent vertices}\;\label{alg:endSecondLoop}
  }
$j:=j+1$\;
}
\caption{Labeling Algorithm (comments enclosed by {\tt /*} and {\tt */})$\strut$}
\label{alg:Labeling}
\end{algorithm}


The algorithm works as follows.  Among the endpoints of all longest
shortest paths, we select one of minimal degree and label it as $1$.
We then go through the vertices in $\Link(1)$ and label them
$2,3,\ldots$, first labeling vertices with the fewest number of edges
connected to unlabeled vertices.  This process is repeated for the
unlabeled vertices connected to vertex~$2$, and vertex $3$, and so on
until every vertex is labeled.  Furthermore, every vertex will be
labeled because we first label everything in $\Linkp(1)$, then label
everything in $\Linkp(2)$ not already labeled, and so on.  Since the
input graph is connected, this process must eventually reach all of
the vertices.  Hence we get a labeling of $G$.

The following lemma explains the function $l$ that appears in 
Algorithm~\ref{alg:Labeling}.

\begin{lemma}
\label{claim:meaningOfFunction}
Let $G$ be a connected graph with the labeling from
Algorithm~\ref{alg:Labeling}.  Then:
\begin{enumerate}
\item $l(1)=0$, and for every $i\in [n]$
with $i>1$, $l(i)=\min(\Link(i))$.
\item If $l(t)<l(s)$, then $t<s$. 
\end{enumerate}
\end{lemma}

\begin{proof}
Algorithm~\ref{alg:Labeling} defines $l(1)=0$.  Now assume $i>1$ and
let $v$ be the vertex assigned the label $i$.  By
lines~\ref{alg:label2} and \ref{alg:function2} of the algorithm, we
need to show that when the label $i$ is assigned to $v$, the variable
$j$ equals $\min(\Link(i))$.  This follows because for any smaller
value $j' < j$, line~\ref{alg:secondLoop} implies that everything in
the neighborhood of $j'$ is labeled before $j'$ is incremented.
However, lines~\ref{alg:Sset}--\ref{alg:beginSecondLoop} show that $v$
is adjacent to $j$ and unlabeled at the start of the loop on
line~\ref{alg:secondLoop}.  Hence $v$ cannot link to any smaller value
of $j$, and since $v$ has label $i$, $j = \min(\Link(i))$ follows.

(2) Suppose that $s,t\in [n]$ satisfy $l(t)<l(s)$.  Since
$l(t)$ (resp.\ $l(s)$) is the value of $j$ when the label $t$
(resp.\ $s$) was assigned in Algorithm~\ref{alg:Labeling},
$l(t)<l(s)$ implies that the label $s$ was assigned later than $t$ in
the algorithm.  Since the labels are assigned in numerical order, we
must have $t < s$. \end{proof}

The labeling produced by Algorithm~\ref{alg:Labeling} allows us to
define the layers $L_N$.  These interact with the function $l$ as follows:

\goodbreak

\begin{lemma}
\label{claim:orderingClaim}
Let $G$ be a connected graph with the labeling from
Algorithm~\ref{alg:Labeling}.  Then:
\begin{enumerate}
\item If $t\in L_N$, then $l(t)\in L_{N-1}$ if $N>0$. 
\item If $t\in L_N$ and $s\in L_M$ with $N<M$, then $t<s$. 
\end{enumerate}
\end{lemma}

\begin{proof} 
We prove (1) and (2) simultaneously by induction on $N \ge 1$ (the
case $N = 0$ of (2) is trivially true).  The first time 
Algorithm~\ref{alg:Labeling} gets to Line~\ref{alg:Sset}, we have $S =
\Link(1)\setminus J = \Link(1) = L_1$.  Every vertex in $S = L_1$, is
labeled during the loop starting on Line~\ref{alg:firstLoop}, so $l(t)
= 1$ for all $t \in L_1$.  Hence (1) holds when $N =1$.  Also, if
$s\in L_M$ with $1<M$, then the vertex $s$ is not labeled at this
stage.  Since labels are assigned in numerical order, we must have
$t<s$ for all $t \in L_1$.  Hence (2) holds when $N = 1$.

Now assume that (1) and (2) hold for $M$ and every $N\leq N_0$.  Given
$t\in L_{N_0+1}$, a shortest path from 1 to $t$ gives $v\in L_{N_0}$
with $v \in \Link(t)$.  Since $l(t)=\min(\Link(t))$ by
Lemma~\ref{claim:meaningOfFunction}(1), we have $l(t)\leq v$.  We have
$l(t) \in L_u$ for some $u$.  If $u > N_0$, then the inductive hypothesis
for (2) would imply $l(t) > v$, which contradicts $l(t)\leq v$.  Hence 
$l(t) \in L_u$ for some $u
\le N_0$.  But $t \in L_{N_0+1}$ and $\{t,l(t)\} \in E(G)$ imply $l(t)
\in L_u$ for $u \ge N_0$ by Lemma~\ref{layerlem}(1).  Hence $l(t) \in
L_{N_0}$, proving (1) for $N_0+1$.

Turning to (2), pick $t\in L_{N_0+1}$ and $s\in L_M$ with $N_0+1<M$.
We just showed that $l(t)\in L_{N_0}$, and Lemma~\ref{layerlem}(1)
implies that $l(s)\in L_{u}$, $u \ge M-1$, since $s \in L_M$.  Then
$N_0 < M-1 \le u$, so our inductive hypothesis, applied to $l(t)\in
L_{N_0}$ and $l(s) \in L_u$, implies $l(t)<l(s)$.  Then $t<s$ by
Lemma~\ref{claim:meaningOfFunction}(2), proving (2) for $N_0+1$.
\end{proof}

\section{Proof of the Main Theorem}
\label{characterize}

We now turn to the main result of the paper.  Theorem~\ref{mainthm}
from the Introduction states that a connected graph is closed if and
only if it is chordal, claw-free and narrow.  One direction is now
proved, since closed graphs are chordal and claw-free by
Proposition~\ref{Hprop}, and connected closed graphs are narrow by
Theorem~\ref{narrowthm}. 

The proof of converse is harder.  The key idea that the labeling
constructed by Algorithm~\ref{alg:Labeling} is closed when the input
graph is chordal, claw-free and narrow.  Thus the proof of
Theorem~\ref{mainthm} will be complete once we prove the following
result.

\begin{theorem}
\label{converse}
Let $G$ be a connected, chordal, claw-free, narrow graph.  Then the
labeling produced by Algorithm~\ref{alg:Labeling} is closed.
\end{theorem}

\begin{proof}
By Proposition~\ref{nbdinterval}, it suffices to show that the
labeling produced by Algorithm~\ref{alg:Labeling} has the property
that for all $m \in V(G) = [n]$, 
\begin{equation}
\label{toprove}
\Linkp(m) \text{ is complete and }
\Linkp(m) = [i+m,i+r_m] \text{ for } r_m = |\Linkp(m)|.
\end{equation}  
We will prove this by induction on $m$.  In \eqref{basecase} below, we
show that \eqref{toprove} holds for $m = 1$, and in \eqref{indstep}
below, we show that if \eqref{toprove} holds for all $1\leq u<m$, then
it also holds for $m$.  Thus, we will be done after proving
\eqref{basecase} and \eqref{indstep}.
 \end{proof}

\subsection{The Base Case} After Algorithm \ref{alg:Labeling} runs on
a chordal, claw-free and narrow graph $G$, the 
base case of the induction in the proof of Theorem~\ref{converse} is
the following assertion:
\begin{equation}
\label{basecase}
\Linkp(1)=[2,1+r],\ r = |\Linkp(1)|,\ \text{and } \Linkp(1) \text{ is
  complete} .
\end{equation}
We will first show that $\Linkp(1)=[2,1+r]$, $r = |\Linkp(1)|$.  The
first time through the the loop beginning on Line~\ref{alg:firstLoop}
in Algorithm~\ref{alg:Labeling}, $j=1$ and $i=2$ and $S=\Link(1)$.
For each vertex in $S$, the loop beginning on
Line~\ref{alg:secondLoop} labels that vertex $i$, removes it from $S$,
and increments $i$.  This continues until $S=\emptyset$, at which
point every vertex in $S$ has been labeled $2,3,\ldots,1+r$, where $r$
is the initial size of $S$.  Hence $\Linkp(1)=\Link(1) = [2,1+r]$.

To prove that $\Linkp(1)$ is complete, there are several cases to
consider.  Pick distinct vertices $s,t\in \Linkp(1)$ and assume that
$\{s,t\}\notin E(G)$.  Note that $s,t\in L_1$ are distance $2$ apart
and therefore $h=\Diam(G)\geq2$.  Our choice of vertex $1$ guarantees
that there is a longest shortest path $P$ with 1 as an endpoint.  Let
$z\in V(G)$ be the other, so that $P=v_0v_1\cdots v_h$, $1 = v_0$ and
$v_h = z$.  Since $v_1 \in V(P)$ is the only vertex of $P$ in $L_1$,
$s$ and $t$ cannot both lie on $P$.

Therefore, either $s\in V(P)$, $t\in V(P)$, or
$s,t\notin V(P)$. We will show that each possibility leads to a
contradiction, proving that $\{s,t\}\in E(G)$.

\vskip 0pt plus 1 pt minus 1 pt

{\sf\bfseries Case 1.} Both $s,t\notin V(P)$.  If $s$ has distance
$h-1$ from $z$, then appending the edge $\{1,s\}$ to a shortest path
from $s$ to $z$ gives a longest shortest path $P'$ from $1$ to $z$
that contains $s$.  Replacing $P$ with $P'$, we get $s \in V(P)$,
which is Case 2 to be considered below.  Similarly, if $t$ has
distance $h-1$ from $z$, then replacing $P$ allows us to assume $t \in
V(P)$, which is also covered by Case 2 below.

Thus we may assume that neither $s$ nor $t$ has distance $h-1$ from
$z$.  Since $d(s,z) < h-1$ would imply $d(1,z) < h$, we conclude that
$s$ has distance $h$ from $z$, and the same holds for $t$.  It follows
that $\{s,v_2\},\{t,v_2\} \notin E(G)$, since otherwise there is a
path shorter than length $h$ from $s$ or $t$ to $z$.

Since the subgraph induced on vertices $1,s,t,v_1$ cannot be a claw,
either $\{v_1,s\}\in E(G)$ or $\{v_1,t\}\in E(G)$ or both. We consider
each possibility separately.

\vskip 0pt plus 1 pt minus 1 pt

{\sf\bfseries Case 1A.} Both $\{v_1,s\}$, $\{v_1,t\}\in E(G)$, as
shown in Figure~\fignum{1}(a) on the next page.  Then the subgraph
induced on $v_2,v_1,s,t$ is a claw, contradicting our assumption of
claw-free.

\begin{figure}[t]
\vskip-10pt
\[
\begin{array}{ccc}
& \begin{tikzpicture}
  \node[vertex] (n1)  at (2,1) {$1\rule[-2.5pt]{0pt}{10pt}$};
  \node[vertex] (s)  at (3.5,3)  {$s\rule[-2.5pt]{0pt}{10pt}$};
  \node[vertex] (t) at (.5,3) {$t\rule[-2.5pt]{0pt}{10pt}$};
  \node[vertex] (v1) at (2,3) {$v_1\rule[-2.5pt]{0pt}{10pt}$}; 
  \node[vertex] (v2) at (2,5) {$v_2\rule[-2.5pt]{0pt}{10pt}$};

  \foreach \from/\to in {n1/v1,n1/t,n1/s,v1/v2,s/v1,t/v1}
 \draw (\from)--(\to);

  \node [right, font = \large] at (4.2,3) {$L_1$};
  \node [right, font=\large] at (4.2,5) {$L_2$};
  \node [right, font = \large] at (4.2,1) {$L_0$};

\node [right] at (2.5,.3) {(a)};

\end{tikzpicture} &\hspace{25pt}\begin{tikzpicture}
  \node[vertex] (n1)  at (2,1) {$1\rule[-2.5pt]{0pt}{10pt}$};
  \node[vertex] (s)  at (3.5,3.2)  {$s\rule[-2.5pt]{0pt}{10pt}$};
  \node[vertex] (t) at (.5,3.2) {$t\rule[-2.5pt]{0pt}{10pt}$};
  \node[vertex] (v1) at (2,3.2) {$v_1\rule[-2.5pt]{0pt}{10pt}$}; 
  \node[vertex] (v2) at (2,5.4) {$v_2\rule[-2.5pt]{0pt}{10pt}$};
  \node[vertex] (t2) at (.5,5.4) {$t_2\rule[-2.5pt]{0pt}{10pt}$};

  \foreach \from/\to in {n1/s,n1/t,n1/v1, v1/v2, s/v1,t/t2}
 \draw (\from)--(\to);

  \node [right, font = \large] at (4.3,3.2) {$L_1$};
  \node [right, font=\large] at (4.3,5.4) {$L_2$};
  \node [right, font = \large] at (4.3,1) {$L_0$};

\node [right] at (2.5,.3) {(b)};

\end{tikzpicture} 
\end{array}
\]
\vskip-8pt
\caption{The portion of the graph relevant to (a) Case 1A and (b) Case
  1B.}
\end{figure}
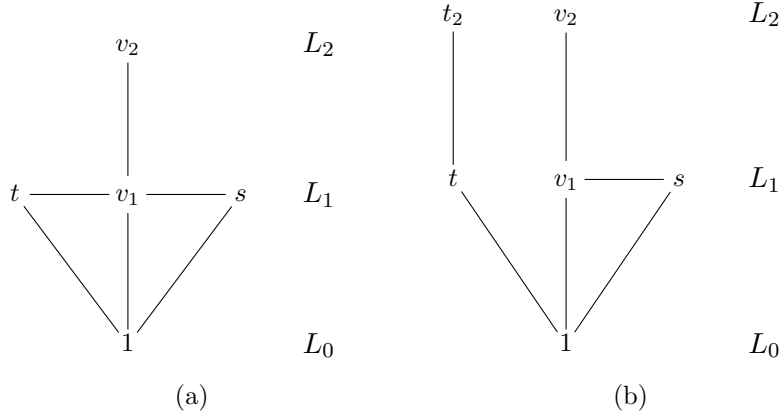

\vskip 0pt plus 1 pt minus 1 pt

{\sf\bfseries Case 1B.} Exactly one of $\{s,v_1\}, \{t,v_1\}$ is in
$E(G)$.  Without loss of generality, we may assume $\{s,v_1\}\in E(G)$
and $\{t,v_1\}\notin E(G)$, as shown in Figure~\fignum{1}(b).  Recall
that $\{t,v_2\}\notin E(G)$ and $t$ is distance $h$ to $z$.

Since $t$ and $1$ are both endpoints of longest shortest paths,
Line~\ref{alg:possible1} of Algorithm~\ref{alg:Labeling} implies that
$\deg(1) \le \deg(t)$.  Since $v_1$ is adjacent to $1$ but not $t$,
there must be at least one $t_2$ adjacent to $t$ but not $1$, i.e.,
$t_2\in \Link(t)$ with $t_2\notin \Link(1)$.

For this $t_2$, it follows that $t_2\in L_2$.  We also have $t_2\neq
v_2$ since $\{t,v_2\}\notin E(G)$.  Furthermore, $\{t_2,s\}\notin
E(G)$, since otherwise we would have the $4$-cycle $t_2s1tt_2$ with no
chords as $\{t_2,1\}$, $\{t,s\}\notin E(G)$.  Similarly,
$\{t_2,v_1\}\notin E(G)$ or else we would have the $4$-cycle
$t_2v_1\hskip-.9pt 1tt_2$ with no chords since $\{t_2,1\}$,
$\{t,v_1\}\notin E(G)$.  Note also that $\{t_2,v_2\}\notin E(G)$,
since otherwise we would have the $5$-cycle $t_2v_2v_1\hskip-.9pt
1tt_2$ with no chords as $\{1,v_2\}$, $\{1,t_2\}$, $\{t,v_2\}$,
$\{t,v_1\}$, $\{t_2,v_1\}\notin E(G)$, contradicting chordal.  Hence
$t_2$ gives Figure~\fignum{1}(b) as an induced subgraph.

Since $G$ is narrow, either $t_2\in V(P)$ or $t_2$ is adjacent to a
vertex of $P$.  However, $t_2\in V(P)$ would imply $t_2=v_2$ since
both lie in $L_2$, contradicting $t_2\neq v_2$.  Thus $\{t_2,v_u\}\in
E(G)$ for some $u>1$.  Since $t_2 \in L_2$ and $v_u \in L_u$, we have
$u \le 3$ by Lemma~\ref{layerlem}(1).  We just proved
$\{t_2,v_2\}\notin E(G)$, so we must have $\{t_2,v_3\} \in E(G)$.
This gives the $6$-cycle $t_2v_3v_2v_1\hskip-.9pt 1tt_2$.  Since
Figure~\fignum{1}(b) is an induced subgraph, the only possible chords
are $\{1,v_3\}$, $\{t,v_3\}$, $\{v_1,v_3\}$, but by
Lemma~\ref{layerlem}(1) none of these are in $E(G)$ since $v_3 \in
L_3$ and $1,t,v_1 \in L_0\cup L_1$. Hence the $6$-cycle has no chords,
contradicting chordal.

{\sf\bfseries Case 2.} $s \in V(P)$ or $t \in V(P)$.  We may assume
$s=v_1$.  Arguing as in Case 1B, there is $t_2\in \Link(t)$ with
$t_2\notin \Link(1)$ and $t_2\in L_2$.  We also have $\{t,v_2\}\notin
E(G)$, since otherwise the $4$-cycle $1sv_2t1$ has no chords as
$\{t,s\}$, $\{1,v_2\}\notin E(G)$.

Since $G$ is narrow, $t_2$ must either be in $P$ or be adjacent to a
vertex in $P$.  However, $t_2 \in V(P)$ would imply $t_2 = v_2$ since
$t_2,v_2 \in L_2$, and the latter would give $\{t,v_2\} = \{t,t_2\}
\in E(G)$, which we just showed to be impossible.  Hence $t_2\notin
V(P)$, so that $\{t_2,v_u\}\in E(G)$ for some $u$.  Note that $u < 4$
by Lemma~\ref{layerlem}(1).  We claim that $u = 3$.

To see why, first note that $\{t_2,s=v_1\}\notin E(G)$, since
otherwise we would have the $4$-cycle $1tt_2v_1\hskip-.9pt 1$ with no
chords as $\{t,s\}$, $\{t_2,1\}\notin E(G)$. We also know that
$\{t_2,v_2\}\notin E(G)$, as otherwise we would have the $5$-cycle
$t_2v_2s1tt_2$ with no chords since $\{t_2,s\}$, $\{t_2,1\}$,
$\{s,t\}$, $\{t,v_2\}$, $\{v_2,1\}\notin E(G)$.  See
Figure~\fignum{2}(a).

Thus we must have $\{t_2,v_3\}\in E(G)$. However, this gives a
$6$-cycle $t_2v_3v_2s1tt_2$ with the same impossible chords as before
along with $\{t,v_3\}$, $\{1,v_3\}$, $\{s,v_3\}$, $\{v_2,t_2\}\notin
E(G)$, as in Figure~\fignum{2}(b).  This contradicts chordal, and
\eqref{basecase} follows.

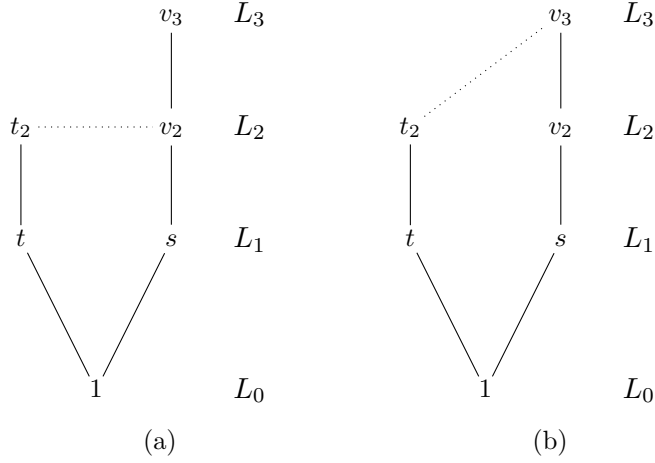
\begin{figure}[t]
\[
\begin{array}{ccc}
\begin{tikzpicture}
  \node[vertex] (n1)  at (2,1) {$1\rule[-2.5pt]{0pt}{10pt}$};
  \node[vertex] (n2)  at (1,3)  {$t\rule[-2.5pt]{0pt}{10pt}$};
  \node[vertex] (n3) at (3,3) {$s\rule[-2.5pt]{0pt}{10pt}$};
 
  \node[vertex] (n4) at (1,4.5){$t_2\rule[-2.5pt]{0pt}{10pt}$};
  \node[vertex] (n5) at (3,4.5){$v_2\rule[-2.5pt]{0pt}{10pt}$};
  \node[vertex] (n6) at (3,6){$v_3\rule[-2.5pt]{0pt}{10pt}$};

  \foreach \from/\to in {n1/n2,n1/n3,n2/n4, n3/n5, n6/n5}
\draw (\from)--(\to);;
 \foreach \from/\to in {n4/n5}
\draw[dotted] (\from)--(\to);

  \node [right, font = \large] at (3.7,3) {$L_1$};
  \node [right, font=\large] at (3.7,4.5) {$L_2$};
  \node [right, font = \large] at (3.7,1) {$L_0$};
  \node [right, font = \large] at (3.7,6) {$L_3$};

  \node [right] at (2.5,.3) {(a)};

\end{tikzpicture}&\hspace{25pt}&
\begin{tikzpicture}
  \node[vertex] (n1)  at (2,1) {$1\rule[-2.5pt]{0pt}{10pt}$};
  \node[vertex] (n2)  at (1,3)  {$t\rule[-2.5pt]{0pt}{10pt}$};
  \node[vertex] (n3) at (3,3) {$s\rule[-2.5pt]{0pt}{10pt}$};
 
  \node[vertex] (n4) at (1,4.5){$t_2\rule[-2.5pt]{0pt}{10pt}$};
  \node[vertex] (n5) at (3,4.5){$v_2\rule[-2.5pt]{0pt}{10pt}$};
  \node[vertex] (n6) at (3,6){$v_3\rule[-2.5pt]{0pt}{10pt}$};

  \foreach \from/\to in {n1/n2,n1/n3,n2/n4, n3/n5, n6/n5}
\draw (\from)--(\to);;
 \foreach \from/\to in {n4/n6}
\draw[dotted] (\from)--(\to);

  \node [right, font = \large] at (3.7,3) {$L_1$};
  \node [right, font=\large] at (3.7,4.5) {$L_2$};
  \node [right, font = \large] at (3.7,1) {$L_0$};
  \node [right, font = \large] at (3.7,6) {$L_3$};

  \node [right] at (2.5,.3) {(b)};

\end{tikzpicture}
\end{array}
\]
\vskip-8pt
\caption{Both (a) and (b) cannot have the dotted edge or the graph has
  a $5$-cycle or $6$-cycle with no chord.} 
\end{figure}

\subsection{The Inductive Step} After Algorithm \ref{alg:Labeling}
runs on a chordal, claw-free and narrow graph $G$, we now prove that
the resulting labeling satsifies the inductive step in the proof of
Theorem~\ref{converse}:
\begin{equation}
\label{indstep}
\begin{aligned}
&\text{If $\Linkp(u)=[u+1,u+r_u]$, $r_u = |\Linkp(u)|$, $\Linkp(u)$ is
    complete, $1\leq u<m$,} \\ &\text{then $\Linkp(m) = [m+1,m+r_m]$,
    $r_m = |\Linkp(m)|$, and $\Linkp(m)$ is complete.}
\end{aligned}
\end{equation}

For the first assertion of \eqref{indstep}, we know that
$\Linkp(m-1)=[m,m-1+r_{m-1}]$ is complete, which implies that
$m+1,\ldots, m-1+r_{m-1}\in \Linkp(m)$.  By analyzing the loop
beginning on Line~\ref{alg:secondLoop} at this stage of
Algorithm~\ref{alg:Labeling}, one finds that every
vertex in $S$ will be labeled with consecutive integers, starting at
$i=m+r_{m-1}$ and continuing until the final vertex in $\Link(m)$ is
labeled $i=m+r_{m-1}+r-1$, where $r$ is the original size of $S$.  It
follows that $\Link(m)$ is an interval of the desired form.

To show that $\Linkp(m)$ is complete, pick $s \ne t$ in $\Linkp(m)$. Let $P= v_0v_1\cdots v_{q-1}v_q$ be
a shortest path from $1 = v_0$ to $v_q = m$, with $v_u \in L_u$ for
all $u$.  Lemmas~\ref{layerlem}(1) and~\ref{claim:orderingClaim}(2) imply that $s,t \in L_q \cup L_{q+1}$.  Hence, $s$
and $t$ are either both distance $q+1$ from 1, both distance $q$ from
1, or one of $s$ and $t$ is distance $q$ from 1 and the other is
distance $q+1$ from 1.  We consider each case separately.

\vskip0pt plus3pt minus2pt

{\sf\bfseries Case 1.} $s,t\in L_{q+1}$.  Then $\{s,v_{q-1}\}$,
$\{t,v_{q-1}\}\notin E(G)$ by Lemma~\ref{layerlem}(1).
Since the subgraph induced on $s,t,m,v_{q-1}$ cannot be a claw, we
must have $\{s,t\} \in E(G)$.

\vskip0pt plus3pt minus2pt

{\sf\bfseries Case 2.} $s,t\in L_q$.  We can assume $s<t$ and choose a
shortest path $P_1=w_0w_1\cdots w_q$ from $1=w_0$ to $w_q = t$ with
$w_u\in L_u$.  Then $w_{q-1}<m$ by Lemma~\ref{claim:orderingClaim}(2),
giving $w_{q-1}<m<s<t$.  Since $t\in \Linkp(w_{q-1})$ and
$\Linkp(w_{q-1})$ is an interval by hypothesis, we have $s\in
\Linkp(w_{q-1})$.  But then $\{s,t\}\in E(G)$ since we are also
assuming that $\Linkp(w_{q-1})$ is complete.

{\vskip0pt plus3pt minus2pt

\sf\bfseries Case 3.} We can assume $s\in L_q$
and $t\in L_{q+1}$, so $s<t$ by
Lemma~\ref{claim:orderingClaim}(2).  We also have $l(m)\leq l(s)$ by
Lemma~\ref{claim:meaningOfFunction}(2) since $m<s$.  We will consider
separately the two possibilities that $l(m)<l(s)$ and $l(m)=l(s)$.

\vskip0pt plus3pt minus2pt

{\sf\bfseries Case 3A.} Suppose that $l(m) < l(s)$.  Then
$\{l(m),s\}\notin E(G)$ since $l(s)=\min(\Link(s))$ by
Lemma~\ref{claim:meaningOfFunction}(1).  We also have $\{l(m),t\}\notin
E(G)$, for otherwise we would have $l(t)\leq l(m)$ since
$l(t)=\min(\Link(t))$.  Then $l(t) \le l(m) < l(s)$, which implies
$t<s$ by Lemma~\ref{claim:meaningOfFunction}(2), contradicting $s<t$.
Since the subgraph induced on $l(m),m,s,t$ cannot be a claw, we must
have $\{s,t\} \in E(G)$.

\vskip0pt plus3pt minus2pt

{\sf\bfseries Case 3B.}  Suppose that $l(m)=l(s)$. We will assume
$\{s,t\} \notin E(G)$ and derive a contradiction.  The equality
$l(m)=l(s)$ means that $m$ and $s$ were both labeled when
$j=l(m)=l(s)$ in the loop starting on Line~\ref{alg:firstLoop} of
Algorithm~\ref{alg:Labeling}.  Consider the moment in the algorithm
when the label $m$ is assigned.  Since $m < s$ and $j=l(m)=l(s)$, this
happens during an iteration of the loop on Line~\ref{alg:secondLoop}
for which $m,s\in S$.  Line~\ref{alg:beginSecondLoop} guarantees that
the vertices assigned the labels $m$ and $s$ satisfy
$|\Link(m)\setminus J| \leq |\Link(s)\setminus J|$.  Since $s$ is not
yet labeled at this point and $s<t$, $t$ is also not yet labeled and
therefore $t\notin J$.  It follows that $t\in \Link(m)\setminus J$ and
$t\notin \Link(s)\setminus J$.  But, in order for $|\Link(m)\setminus
J| \leq |\Link(s)\setminus J|$ to hold, there must be $s_2\in
\Link(s)$ with $s_2>m$ and $s_2\notin J$ and $s_2\notin \Link(m)$.

Let us study $s_2$.  If $\{s_2,l(m)\}\in E(G)$, then $s_2\in
\Linkp(l(m))$.  But we also have $m \in \Linkp(l(m))$.  Since
$l(m)<m$, $\Linkp(l(m))$ is complete by the hypothesis of \eqref{indstep},
so we would have $\{m,s_2\}\in E(G)$.  This contradicts our choice of
$s_2$.  Hence $\{s_2,l(m)\}\notin E(G)$.  We also have
$\{s_2,t\}\notin E(G)$, since otherwise the $4$-cycle $s_2tmss_2$
would have no chords as $\{s_2,m\},\{s,t\}\notin E(G)$.  Also, since
$m\in L_q$, Lemma~\ref{claim:orderingClaim}(1) implies that
$j=l(m)=l(s)\in L_{q-1}$.  

We claim that $s_2 \in L_{q+1}$.  Lemma~\ref{layerlem}(1), $s \in
L_q$, and $s_2 > m \in L_q$ imply that $s_2 \in L_q$ or $L_{q+1}$.  If
$s_2 \in L_q$, then $l(s_2) \in L_{q-1}$ by
Lemma~\ref{claim:orderingClaim}(2).  From here, $m \in L_q$ implies $m
> l(s_2)$ by Lemma~\ref{claim:orderingClaim}(2).  Hence we have
$l(s_2) < m < s_2$.  The hypothesis of \eqref{indstep} implies that
$\Linkp(l(s_2))$ is complete and is an interval.  Since $s_2 \in
\Linkp(l(s_2))$, it follows that $m \in \Linkp(l(s_2))$, which
contradicts our choice of $s_2$.  Hence $s_2 \in L_{q+1}$ and we have
Figure~\fignum{3}(a).  


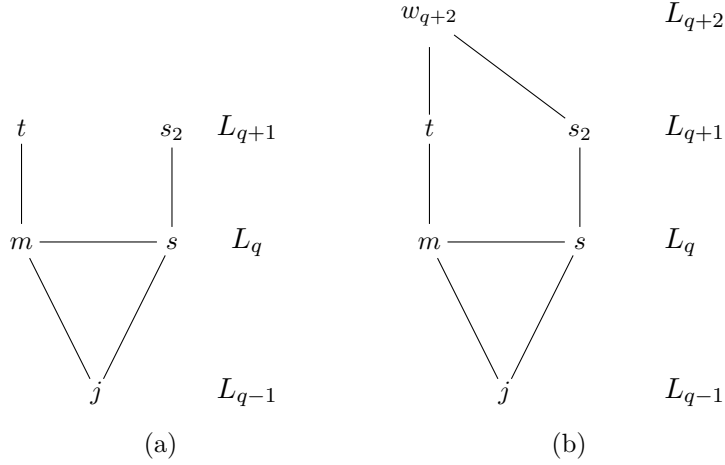
\begin{figure}[t]
\[
\begin{array}{ccc}
\begin{tikzpicture}
  \node[vertex] (n1)  at (2,1) {$j\rule[-2.5pt]{0pt}{10pt}$};
  \node[vertex] (n2)  at (1,3)  {$m\rule[-2.5pt]{0pt}{10pt}$};
  \node[vertex] (n3) at (3,3) {$s\rule[-2.5pt]{0pt}{10pt}$};
 
  \node[vertex] (n4) at (1,4.5){$t\rule[-2.5pt]{0pt}{10pt}$};
  \node[vertex] (n5) at (3,4.5){$s_2\rule[-2.5pt]{0pt}{10pt}$};

  \foreach \from/\to in {n1/n2,n1/n3,n2/n3,n2/n4, n3/n5}
\draw (\from)--(\to);;

  \node [vertex, font = \large] at (4,3) {$L_q$};
  \node [vertex, font=\large] at (4,4.5) {$L_{q+1}$};
  \node [vertex, font = \large] at (4,1) {$L_{q-1}$};

  \node [right] at (2.5,.3) {(a)};

\end{tikzpicture}&\hspace{25pt}&
\begin{tikzpicture}
  \node[vertex] (n1)  at (2,1) {$j\rule[-2.5pt]{0pt}{10pt}$};
  \node[vertex] (n2)  at (1,3)  {$m\rule[-2.5pt]{0pt}{10pt}$};
  \node[vertex] (n3) at (3,3) {$s\rule[-2.5pt]{0pt}{10pt}$};
 
  \node[vertex] (n4) at (1,4.5){$t\rule[-2.5pt]{0pt}{10pt}$};
  \node[vertex] (n5) at (3,4.5){$s_2\rule[-2.5pt]{0pt}{10pt}$};
  \node[vertex] (n6) at (1,6){$w_{q+2}$};

  \foreach \from/\to in {n1/n2,n1/n3,n2/n3,n2/n4, n3/n5, n4/n6, n6/n5}
\draw (\from)--(\to);;

  \node [right, font = \large] at (4,3) {$L_q$};
  \node [right, font=\large] at (4,4.5) {$L_{q+1}$};
  \node [right, font = \large] at (4,1) {$L_{q-1}$};
  \node [right, font = \large] at (4,6) {$L_{q+2}$};

  \node [right] at (2.5,.3) {(b)};

\end{tikzpicture}
\end{array}
\]
\vskip-10pt
\caption{(a) Induced subgraph on $j=l(m)=l(s)$, $m$, $s$, $t$, $s_2$
  and (b) $5$-cycle with no chords.} 
\end{figure}

Let $z$ be a vertex of distance $h=\Diam(G)$ from $1$ and pick a
longest shortest path $P_2=w_0w_1\cdots w_h$ from $1=w_0$ to $w_h =z$,
so $w_u\in L_u$.  Since $G$ is narrow, $t$ and $s_2$ must each either
be in $P_2$ or be adjacent to a vertex in $P_2$.  We will consider
each of these cases. 

\emph{First}, suppose that $t\in V(P_2)$.  Then $t\in L_{q+1}$ implies
that $t=w_{q+1}$.  Since $l(m) \in L_{q-1}$, there is a path of length
$q-1$ connecting $1$ to $l(m)$.  Using $t = w_{q+1}$, it follows that
$P_3=1\cdots l(m) m t w_{q+2}\cdots z$ is a path of length
$h=\Diam(G)$.  Since $G$ is narrow, $s_2$ must be adjacent to some
vertex $P_3$.  Then $\{s_2,t\}$, $\{s_2,m\},\notin E(G)$ and
Lemma~\ref{layerlem}(1) imply that $\{s_2,w_{q+2}\} \in E(G)$.  This
gives the $5$-cycle $mss_2w_{q+2}tm$ with no chords since $\{s_2,t\}$,
$\{m,s_2\}$, $\{s,t\}\notin E(G)$ and $\{w_{q+2},m\}$,
$\{w_{q+2},s\}\notin E(G)$ since $w_{q+2}\in L_{q+2}$ but $s,m\in
L_q$.  See Figure~\fignum{3}(b).  Hence we have a contradiction since
$G$ is chordal.

\emph{Second}, suppose that $s_2\in V(P_2)$.  Then $s_2=w_{q+1}$.
Arguing as in the \emph{First}, we arrive at Figure~\fignum{3}(b) with
the same $5$-cycle with no chords, again a contradiction.

\emph{Third}, suppose that $s_2,t\notin V(P_2)$.  First note that $P_2$
was an arbitrary longest shortest path starting at $1$.  Thus the
above \emph{First} and \emph{Second} give a contradiction whenever
$s_2$ or $t$ are on \emph{any} longest shortest path starting at $1$.
Hence we may assume that $s_2$ and $t$ are not on any shortest path of
length $h$ starting at $1$.  

Since $G$ is narrow, $s_2 \in L_{q+1}$ is adjacent to a vertex of
$P_2$, which must be $w_q$, $w_{q+1}$, or $w_{q+2}$ by
Lemma~\ref{layerlem}(1).  However, if $\{s_2,w_{q+2}\} \in E(G)$, then
we would get a path of length $h$ from $1$ to $z$ by taking any shortest
path from 1 to $s_2$, followed by $\{s_2,w_{q+2}\}$, and then
continuing along  $P_2$ from $w_{q+2}$ to $z$.  This longest shortest
path starts at $1$ and contains $s_2$, contradicting the previous
paragraph.  Hence  $\{s_2,w_{q+2}\} \notin E(G)$ and
$s_2$ must be adjacent to $w_q$ or $w_{q+1}$, and 
the same is true for $t$ by a similar argument.

In fact, we must have $\{s_2,w_q\} \in E(G)$, since otherwise
$\{s_2,w_{q+1}\} \in E(G)$ and the subgraph induced on
$w_q,w_{q+1},w_{q+2},s_2$ would be a claw.  A similar argument shows
that $\{t,w_q\} \in E(G)$.  Since $w_{q-1}\in L_{q-1}$ and $s_2,t\in
L_{q+1}$, this implies that the subgraph induced on
$t,s_2,w_q,w_{q-1}$ is a claw, again contradicting claw-free.  This
final contradiction completes the proof of \eqref{indstep}, and
Theorem~\ref{converse} is proved.

\begin{remark}
\label{remark:Ching}
In \eqref{basecase} and \eqref{indstep}, the chordal hypothesis is
applied only to cycles of length $4$, $5$, or $6$.  Hence, in
Theorem~\ref{mainthm} and Corollary~\ref{cormainthm}, we can replace
chordal with the weaker hypothesis that all cycles of length $4$, $5$,
or $6$ have a chord.
\end{remark}

\section*{Acknowledgements}

Theorem~\ref{mainthm} is based on the senior honors thesis of the
second author, written under the direction of the first author.  We
are grateful to Amherst College for the Post-Baccalaureate Summer
Research Fellowship that supported the writing of this paper.  Thanks
also to Michael Ching for Remark~\ref{remark:Ching}.

\end{document}